\documentclass{amsart}
\usepackage{amsmath,amscd,xypic,amssymb,combelow,color,enumitem}

\xyoption{all}
\CompileMatrices

\makeatletter
\@addtoreset{equation}{section}
\makeatother

\newtheorem{theorem}[subsection]{Theorem}

\newtheorem{proposition}[subsection]{Proposition}

\newtheorem{definition}[subsection]{Definition}

\newtheorem{claim}[subsection]{Claim}

\def\loccitt{\emph{loc. cit.}}
\def\loccit{\emph{loc. cit. }}

\def\fgl{{\mathfrak{gl}}}

\def\fV{{\mathfrak{V}}}
\def\fW{{\mathfrak{W}}}
\def\fZ{{\mathfrak{Z}}}

\def\BA{{\mathbb{A}}}
\def\BC{{\mathbb{C}}}

\def\BN{{\mathbb{N}}}

\def\BP{{\mathbb{P}}}

\def\BQ{{\mathbb{Q}}}
\def\BZ{{\mathbb{Z}}}

\def\CE{{\mathcal{E}}}
\def\CF{{\mathcal{F}}}

\def\CL{{\mathcal{L}}}
\def\CM{{\mathcal{M}}}

\def\CO{{\mathcal{O}}}

\def\CS{{\mathcal{S}}}

\def\CU{{\mathcal{U}}}
\newcommand\CV{{\mathcal{V}}}

\def\tCF{\widetilde{\CF}}

\def\tCU{\widetilde{\CU}}

\def\Hom{\textrm{Hom}}

\def\e{\varepsilon}

\def\td{\tilde{d}}

\def\vs{\varsigma}

\def\pt{\textrm{pt}}
\def\and{\textrm{ }\&\textrm{ }}

\def\sym{\textrm{Sym}}

\def\op{\text{op}}

\def\tCF{\widetilde{\CF}}

\def\sym{\textrm{Sym}}

\def\nn{{{\BN}}^n}
\def\zz{{{{\mathbb{Z}}}^n}}

\def\UU{{U_{q,\oq}(\ddot{\fgl}_n)}}

\def\UUm{{U^-_{q,\oq}(\ddot{\fgl}_n)}}
\def\UUp{{U^+_{q, \oq}(\ddot{\fgl}_n)}}
\def\UUpm{{U^\pm_{q, \oq}(\ddot{\fgl}_n)}}
\def\UUl{{U^\leq_{q, \oq}(\ddot{\fgl}_n)}}
\def\UUg{{U^\geq_{q, \oq}(\ddot{\fgl}_n)}}

\def\UUup{{U_{q,\oq}^\uparrow(\ddot{\fgl}_n)}}

\def\wUUup{{\widehat{U}_{q,\oq}^\uparrow(\ddot{\fgl}_n)}}

\def\bd{{\mathbf{d}}}

\def\bk{{\mathbf{k}}}
\def\bl{{\mathbf{l}}}
\def\br{{\mathbf{r}}}

\def\bs{{\boldsymbol{\vs}}}

\def\la{{\lambda}}
\def\sq{{\square}}
\def\bsq{{\blacksquare}}

\def\lamu{{\lambda \backslash \mu}}

\def\bla{{\boldsymbol{\la}}}
\def\bmu{{\boldsymbol{\mu}}}

\def\blamu{{\boldsymbol{\lamu}}}

\def\syt{{\text{SYT}}}

\def\oQ{{\overline{Q}}}
\def\oq{{\overline{q}}}

\def\bari{\bar{i}}
\def\barj{\bar{j}}
\def\bark{\bar{k}}

\def\wa{\widehat{a}}
\def\wb{\widehat{b}}

\def\barf{\bar{f}}

\def\zzz{\frac {\BZ^2}{(n,n)\BZ}}
\def\fff{\mathbb{Q}(q,\oq^{\frac 1n})}

\def\tbd{\widetilde{\bd}}

\begin{document}
	
\title[Toward AGT for parabolic sheaves]{\large{\textbf{Toward AGT for parabolic sheaves}}}
\author[Andrei Negu\cb t]{Andrei Negu\cb t}
\address{MIT, Department of Mathematics, Cambridge, MA, USA}
\address{Simion Stoilow Institute of Mathematics, Bucharest, Romania}
\email{andrei.negut@gmail.com}
	
\maketitle
	
	
\begin{abstract} We construct explicit elements $W_{ij}^k$ in (a completion of) the shifted quantum toroidal algebra of type $A$, and show that these elements act by 0 on the $K$--theory of moduli spaces of parabolic sheaves. We expect that the quotient of the shifted quantum toroidal algebra by the ideal generated by the elements $W_{ij}^k$ will be related to $q$--deformed $W$--algebras of type $A$ for arbitrary nilpotent, which would imply a $q$--deformed version of the AGT correspondence between gauge theory with surface operators and conformal field theory.
	
\end{abstract}

\section{Introduction}

\subsection{} A standard, but very fruitful, problem in geometric representation theory is to realize the cohomology groups of moduli spaces as modules for appropriate algebras. One of the classical instances of this principle was Grojnowski's (\cite{G}) and Nakajima's (\cite{Nak}) discovery of the fact that Heisenberg algebras act on the cohomology groups of Hilbert schemes of points on surfaces. This led to a representation--theoretic interpretation of the Betti numbers of Hilbert schemes, and spurred a great deal of innovation and interest into quiver varieties (\cite{Nak 2}). \\

\noindent A slightly more recent, although related, development in the field was the following mathematical formulation of the Alday-Gaiotto-Tachikawa (AGT, \cite{AGT}) correspondence between certain gauge theories and certain conformal field theories:
\begin{equation}
\label{eqn:agt h}
W(\fgl_r) \curvearrowright H^*_{\text{equiv}}(\text{moduli of rank }r \text{ sheaves on } \BA^2)
\end{equation} 
where the left-hand side is the universal enveloping algebra of the $W$--algebra of $\fgl_r$ (which is just the Heisenberg algebra when $r=1$ and the Heisenberg-Virasoro algebra when $r=2$) and the moduli space in the right-hand side parametrizes torsion-free rank $r$ coherent sheaves on $\BP^2$ which are framed along $\infty$. Relation \eqref{eqn:agt h} was proved in \cite{MO} and \cite{SV}. \\

\noindent In \cite{W}, we introduced a new approach to proving \eqref{eqn:agt h}, using explicit correspondences to define the $W$--algebra action. This allowed us to deform the AGT correspondence from cohomology to algebraic $K$--theory: 
\begin{equation}
\label{eqn:agt k}
_qW(\fgl_r) \curvearrowright K_{\text{equiv}}(\text{moduli of rank }r \text{ sheaves on } \BA^2)
\end{equation} 
where the left-hand side is now the deformed $W$--algebra, as defined by \cite{AKOS, FF, FR1}. Our intersection-theoretic approach also allowed us to generalize the statement to other smooth projective surfaces $S$ instead of $\BA^2$ (\cite{W surf, Hecke}). The strategy to proving \eqref{eqn:agt k} is to take the action (\cite{FT1, SV2}) of the quantum toroidal algebra $U_{q,\oq}(\ddot{\fgl}_1)$ on the right-hand side of this equation, to construct an explicit quotient:
\begin{equation}
\label{eqn:bk 1}
\widehat{U}^\uparrow_{q,\oq}(\ddot{\fgl}_1) \twoheadrightarrow {_qW(\fgl_r)}
\end{equation}
\footnote{The superscript $\uparrow$ refers to a certain half of the quantum toroidal algebra, which we will review in Subsection \ref{sub:asi}. The hat refers to a completion which we will review in Subsection \ref{sub:hat}} and to show that the action on $K$--theory factors through this quotient. In \cite{BK, RS}, the authors showed that finite $W$--algebras in type $A$ are quotients of shifted Yangians, and \eqref{eqn:bk 1} can be perceived as an affine, $q$--deformed version of this statement in the case of principal nilpotent. Therefore, a natural question to ask is whether one can generalize the statement \eqref{eqn:bk 1} to the deformed $W$--alegbras corresponding to arbitrary type $A$ nilpotent, a.k.a. Jordan normal form $\br = (r_1,...,r_n)$:
\begin{equation}
\label{eqn:bk 2}
\widehat{U}^\uparrow_{q,\oq}(\ddot{\fgl}_n)^\br \stackrel{?}\twoheadrightarrow {_qW(\fgl_{r_1+...+r_n})^\br}
\end{equation}
where the left-hand side is the shifted quantum toroidal algebra associated to $\br$ (we recall its definition in Subsection \ref{sub:shift}). However, at this point we run into a snag: the right-hand side of \eqref{eqn:bk 2} has not yet been defined for general nilpotent (although we plan to tackle the case of rectangular nilpotent $\br = (r,...,r)$ in \cite{Upcoming}). \\

\noindent Therefore, in the present paper, we take a geometric approach to understanding \eqref{eqn:bk 2}. The quantum toroidal algebra in the left-hand side of this relation is known (\cite{FT2}) to act on the $K$--theory of the following moduli space:
\begin{equation}
\label{eqn:new}
U_{q,\oq}(\ddot{\fgl}_n)^\br \curvearrowright K^\br = K_{\text{equiv}}\left( \begin{array}{c} \text{moduli of sheaves on }\BA^2 \text{ with parabolic} \\ \text{structure of type }r_1,...,r_n \text{ along } \BA^1 \end{array} \right)
\end{equation}
(when $\br = (1,...,1)$, the moduli space above is the well-known affine Laumon space). \\

\begin{theorem}
\label{thm:main}

There exist elements: 
\begin{equation}
\label{eqn:w first}
W_{ij}^k \in \widehat{U}^\uparrow_{q,\oq}(\ddot{\fgl}_n)^\br
\end{equation}
defined in \eqref{eqn:w gen} for all:
$$
(i,j) \in \zzz, \ k \in \BZ
$$
such that $W_{ij}^k$ acts by 0 in the right-hand side of \eqref{eqn:new} if $k > r_j$. \\

\end{theorem}

\noindent Our hope (too vague at this point to be called a conjecture) is that the quotient:
\begin{equation}
\label{eqn:quotient}
\widehat{U}^\uparrow_{q,\oq}(\ddot{\fgl}_n)^\br \Big / (W_{ij}^k)_{(i,j) \in \zzz}^{k > r_j} 
\end{equation}
is closely related to the (still undefined) $q$--deformed $W$--algebra of type $\fgl_{r_1+...+r_n}$ corresponding to nilpotent of type $\br$. If \eqref{eqn:quotient} itself would deserve the name ``$q$--deformed $W$--algebra", then Theorem \ref{thm:main} would establish the $q$--deformed AGT correspondence in the presence of surface operators. This hope holds for $n=1$ (\cite{W}), and for finite $W$--algebras and arbitrary nilpotent (\cite{BFFR, BK}). Meanwhile, in the mathematical physics literature, the AGT correspondence with surface operators has already been intensely studied (at least in the non-deformed case) with emphasis on computing partition functions (\cite{AT, AFKMY, KT, N, Wyl}, and others). \\

\noindent We expect Theorem \ref{thm:main} to generalize to other smooth surfaces $S$ instead of $\BA^2$, where the parabolic structure is taken along any smooth curve $C \subset S$ instead of $\BA^1 \subset \BA^2$, following the proof of \cite{W surf} and \cite{Hecke} in the principal nilpotent case. Although we do not foresee major difficulties in the overall strategy, we make no claims about the technical statements involved in this generalization. \\

\noindent As the present paper is meant as an introduction to our viewpoint on the connection between deformed $W$--algebras and moduli spaces of parabolic sheaves, we prioritize presenting the main ideas over completeness of the proofs. More concretely, many results herein were proved in special cases (either $n=1$ or $r_1=...=r_n=1$) in the author's previous works \cite{W,Tor,Aff,Surf,W surf,Hecke}. Whenever this happens and the proof of the general case is completely analogous to the special case, we will only provide a sketch of a proof and bibliographical references. \\

\noindent I would like to thank Michael Finkelberg, Davide Gaiotto, Ryosuke Kodera, Francesco Sala, Yuji Tachikawa and Alexander Tsymbaliuk for numerous very interesting discussions on shifted quantum toroidal algebras, moduli spaces of parabolic sheaves and AGT. I gratefully acknowledge the NSF grants DMS--1760264 and DMS--1845034, as well as support from the Alfred P. Sloan Foundation. \\

\section{The shifted quantum toroidal algebra}

\subsection{} 
\label{sub:shift}

Many variables in the present paper will be assigned a ``color" $\in \BZ$. We will consider color-dependent rational functions, such as:
\begin{equation}
\label{eqn:zeta}
\zeta \left( \frac zw \right) = \left( \frac {zq \oq^{2\left \lceil \frac {i-j}n \right \rceil} - wq^{-1}}{z\oq^{2\left \lceil \frac {i-j}n \right \rceil} - w} \right)^{\delta_{\bari}^{\barj} - \delta_{\overline{i+1}}^{\barj}}
\end{equation}
where $z$ and $w$ are variables of color $i$ and $j$, respectively. Above and hereafter, we write $\bark$ for the residue class of the integer $k$ in the set $\{1,...,n\}$. \\

\begin{definition}
	
For $\br \in \nn$, the \textbf{shifted quantum toroidal algebra} is:
$$
\UU^{\br} = \fff \Big \langle e_{i,k}, f_{i,k}, \psi^{\pm}_{i,k'}\Big \rangle^{i \in \{1,...,n\}}_{k \in \BZ, k' \in \BN \sqcup 0} \Big/\text{relations \eqref{eqn:rel tor 0}--\eqref{eqn:rel tor 5}}
$$
where for $e_i(z) = \sum_{k \in \BZ} \frac {e_{i,k}}{z^k}$, $f_i(z) = \sum_{k \in \BZ} \frac {f_{i,k}}{z^k}$, $\psi_i^\pm(z) = (-1)^{r_i\delta_\pm^+} \sum_{k=0}^\infty \frac {\psi_{i,k}^\pm}{z^{\pm k}}$, set:
\begin{equation}
\label{eqn:rel tor 0}
[\psi_{i,k}^\pm, \psi_{i',k'}^{\pm'}] = 0
\end{equation}
\begin{align}
e_i(z) e_j(w) \zeta \left( \frac wz \right) &= e_j(w) e_i(z) \zeta \left( \frac zw \right) \label{eqn:rel tor 1} \\
f_i(z) f_j(w) \zeta \left( \frac zw \right) &= f_j(w) f_i(z) \zeta \left( \frac wz \right) \label{eqn:rel tor 2} 
\end{align}
\begin{align} 
e_i(z) \psi^\pm_j(w) &= \psi^\pm_j(w) e_i(z) \zeta \left( \frac zw \right)^{-1} \label{eqn:rel tor 3} \\
f_i(z) \psi^\pm_j(w) &= \psi^\pm_j(w) f_i(z) \zeta \left( \frac zw \right) \label{eqn:rel tor 4} 
\end{align} 
and:
\begin{equation}
\label{eqn:rel tor 5}
\Big[ e_i(z), f_j(w) \Big] = \delta_i^j  \delta \left( \frac zw \right) (q^{-2} - 1) \left[ z^{r_{i+1}-r_i} \frac{ \psi^+_{i+1}(zq^2)}{\psi^+_{i}(z)} - \frac{ \psi^-_{i+1}(wq^2)}{\psi^-_{i}(w)} \right]
\end{equation}
In all relations above, $z$ and $w$ are thought of as variables of colors $i$ and $j$, respectively. We allow $i,j$ to be arbitrary integers, by setting: 
\begin{align}
&e_i(z) = e_{i-n}(z\oq^2) \cdot \oq^{-r_{i+1}} \label{eqn:period 1} \\
&f_i(z) = f_{i-n}(z\oq^2) \cdot \oq^{r_i} \label{eqn:period 2} \\
&\psi^\pm_i(z) = \psi^\pm_{i-n}(z\oq^2) \cdot (q^{r_1+...+r_n} \oq^{r_i})^{\pm 1} \label{eqn:period 3}
\end{align}
for all $i \in \BZ$. \\

\end{definition} 

\noindent The shifted quantum toroidal algebra only depends on the successive differences of the natural numbers $r_i$. With this in mind, the usual (or unshifted) quantum toroidal algebra is simply the case when $\br = (r,...,r)$ for any $r \in \BN$:
\begin{equation}
\label{eqn:unshifted}
\UU = \UU^{(r,...,r)}
\end{equation}
We note that $\UU$ is customarily defined with two central elements (which would be specialized to $q^{nr}\oq^r$ and $1$ in our notation) and with the extra relation $\psi_{i,0}^+ \psi_{i,0}^- = 1$, which we do not impose in the present paper. Thus, we slightly abuse terminology by calling \eqref{eqn:unshifted} the ``quantum toroidal algebra". \\

\subsection{} 
\label{sub:subalgebras} 

We observe that the subalgebras:
\begin{align}
&\UUp^{\br} = \Big \langle e_{i,k} \Big \rangle^{i \in \{1,...,n\}}_{k \in \BZ} \subset \UU^{\br} \label{eqn:sub 1} \\
&\UUm^{\br} = \Big \langle f_{i,k} \Big \rangle^{i \in \{1,...,n\}}_{k \in \BZ} \subset \UU^{\br} \label{eqn:sub 2} \\
&\UUg^{\br} = \Big \langle e_{i,k}, \psi_{i',k'}^+ \Big \rangle^{i,i' \in \{1,...,n\}}_{k \in \BZ, k' \in \BN \sqcup 0} \subset \UU^{\br} \label{eqn:sub 3} \\
&\UUl^{\br} = \Big \langle f_{i,k}, \psi_{i',k'}^- \Big \rangle^{i,i' \in \{1,...,n\}}_{k \in \BZ, k' \in \BN \sqcup 0} \subset \UU^{\br} \label{eqn:sub 4}
\end{align}
do not actually depend on the shift, and so are isomorphic to the corresponding subalgebras in the unshifted ($\br = (r,...,r)$) case. In fact, \eqref{eqn:sub 3} and \eqref{eqn:sub 4} are bialgebras with respect to a certain topological coproduct, and there exists a pairing:
$$
\UUg \otimes \UUl \stackrel{\langle \cdot, \cdot \rangle}\longrightarrow \fff
$$
With this in mind, the unshifted quantum toroidal algebra is the Drinfeld double:
\begin{equation}
\label{eqn:dd 0}
\UU \cong \UUg \otimes \UUl 
\end{equation}
where the commutation relations between the factors in \eqref{eqn:dd 0} are governed by:
\begin{equation}
\label{eqn:drinfeld 1}
\langle a_1, b_1 \rangle a_2b_2 = b_1 a_1 \langle a_2, b_2 \rangle
\end{equation}
for all $a \in \UUg, b \in \UUl$. We refer the reader to \cite{Tor} for more details on the claims made above. An analogous statement can be made in the shifted case, in that we have the following isomorphism of vector spaces:
$$
\UU^{\br} \cong  \UUg \otimes \UUl
$$
but the commutation relations between the two subalgebras are now governed by:
\begin{equation}
\label{eqn:drinfeld 2}
\langle a_1, b_1 \rangle a_2b_2 = b_1 a_1 \langle T^\br(a_2), b_2 \rangle
\end{equation}
for all $a \in \UUg, b \in \UUl$, where:
$$
T^\br : \UUg \longrightarrow \UUg
$$
is the bialgebra automorphism given by:
$$
T^\br \left( e_{i,k} \right) = e_{i,k+r_{i+1}-r_i} \qquad \qquad T^\br \left(\psi_{i,k'}^+ \right) = \psi_{i,k'}^+
$$
for all $i\in \{1,...,n\}$ and for all $k \in \BZ$, $k' \in \BN \sqcup 0$. \\

\subsection{} Consider the abelian group $\nn$ with the following pairing:
\begin{equation}
\label{eqn:pairing}
\nn \times \nn \stackrel{\langle \cdot , \cdot \rangle}\longrightarrow \BZ, \qquad \qquad \langle \bk, \bl \rangle = \sum_{i=1}^n k_il_i - k_{i-1}l_i
\end{equation}
where we set $k_0 = k_n$. We will encounter the following elements of $\nn$:
$$
\bs^i = \underbrace{(0,...,0,1,0,...,0)}_{1 \text{ on }i\text{--th position}}
$$
and we define:
$$
[i;j) = - [j;i) = \begin{cases} \bs^i + ... + \bs^{j-1} &\text{if } i\leq j \\ -\bs^j - ... - \bs^{i-1} &\text{if } i>j \end{cases}
$$
for all $(i,j) \in \BZ^2$. The algebra $\UU$ is graded by $\zz \times \BZ$, with:
\begin{align*}
&\deg e_{i,k} = (\bs^i,k) \qquad \qquad \deg \psi_{i,k}^+ = (0, k) \\
&\deg \psi_{i,k}^- = (0, r_i - k) \qquad \ \deg f_{i,k} = (-\bs^i,r_{i+1} - r_i + k) 
\end{align*}

\subsection{} As we already noted, the subalgebras $\UUp$ and $\UUm$ do not depend on the shift. We will now recall their shuffle algebra interpretation (\cite{E, FO, Tor}). \\

\begin{definition}
\label{def:shuf classic}
	
Consider the vector space:
\begin{equation}
\label{eqn:shuf classic}
\CV^+ = \bigoplus_{\bd = (d_1,...,d_n) \in \nn} \fff(...,z_{i1},...,z_{id_i},...)^{\emph{Sym}}
\end{equation}
and endow it with an associative algebra structure, by setting $R*R'$ equal to:
$$
\emph{Sym} \left[ \frac {R(...,z_{i1},...,z_{id_i},...)}{d_1!...d_n!} \frac {R'(...,z_{i,d_i+1},...,z_{i,d_i+d'_i},...)}{d_1'! ... d_n'!} \prod_{i,i' = 1}^n \prod^{1 \leq a \leq d_i}_{d_{i'} < a' \leq d_{i'}+d'_{i'}} \zeta \left( \frac {z_{ia}}{z_{i'a'}} \right) \right] 
$$
where $\emph{Sym}$ denotes symmetrization with respect to the variables $z_{i1},...,z_{i,d_i+d_i'}$ for each $1 \leq i \leq n$ separately, and $z_{ia}$ is considered to be a variable of color $i$ for all $a$. \\
	
\end{definition}

\noindent Let $\CS^+ \subset \CV^+$ be the subalgebra generated by $\{ z_{i1}^k \}^{1\leq i \leq n}_{k \in \BZ}$. It is graded by $\nn \times \BZ$: 
$$
\deg R(...,z_{i1},...,z_{id_i},...) = (\bd, k)
$$
if $\bd = (d_1,..,d_n)$, while $k$ is the homogeneous degree of $R$. We will refer to $\bd$ as the ``horizontal degree" and to $k$ as the ``vertical degree" of $R$, respectively. In \cite{Tor}, we constructed the following elements of $\CS^+$ for any $i<j$ and Laurent polynomial $M$:
\begin{equation}
\label{eqn:shuf} 
A^M_{[i;j)} = \sym \left[ \frac {M(z_i,...,z_{j-1})}{\left(1 - \frac {z_i q^2}{z_{i+1}}\right) ... \left(1 - \frac {z_{j-2} q^2}{z_{j-1}}\right)} \prod_{i\leq a < b < j} \zeta \left( \frac {z_b}{z_a} \right)  \right] 
\end{equation}
In order to think of the RHS of \eqref{eqn:shuf} as elements of \eqref{eqn:shuf classic}, we relabel the variables $z_i,...,z_{j-1}$ according to the following rule for all $s \in \{i,...,j-1\}$:
$$
z_s \quad \leadsto \quad z_{\bar{s},\left \lfloor \frac {s-i}n \right \rfloor +1} \oq^{-2\left \lfloor \frac {s-1}n \right \rfloor}
$$
The degree of the element \eqref{eqn:shuf} is $([i;j), \text{hom deg }M)$. \\

\subsection{} 
\label{sub:up} 

Let $\CS^{-} = \CS^{+,\op}$. It was shown in \cite{E} that there exist algebra isomorphisms: 
$$
\UUp \stackrel{\Upsilon^+}\cong \CS^+ \qquad \text{and} \qquad \UUm \stackrel{\Upsilon^-}\cong \CS^{-} 
$$
and we will write:
\begin{equation}
\label{eqn:elements}
e_{[i;j)}^{M} = \Upsilon^+ \left(A_{[i;j)}^M \right) \qquad \text{and} \qquad f_{[i;j)}^{M} = \Upsilon^- \left(A_{[i;j)}^M \right)
\end{equation}
for any Laurent polynomial $M$. The following specific $M$'s will feature in the present paper, for any $k \in \BZ$:
\begin{equation}
\label{eqn:el k}
e_{[i;j),k} = e_{[i;j)}^{M(z_i,...,z_{j-1}) \mapsto z_i^k} 
\end{equation} 
\begin{equation}
\label{eqn:el mu}
e^{(k)}_{[i;j)} = e_{[i;j)}^{M(z_i,...,z_{j-1}) \mapsto \prod_{a=i}^{j-1} (z_a \oq^{\frac {2a}n})^{\left \lceil \frac {(a-i+1)k}{j-i} \right \rceil - \left \lceil \frac {(a-i)k}{j-i} \right \rceil}}
\end{equation}
and the analogous notations for the $f$'s. We will write: 
\begin{equation}
\label{eqn:rule}
e_{[i;i),k} = \delta_k^0
\end{equation}
Consider the subalgebra:
\begin{equation}
\label{eqn:top}
U_{q,\oq}^{\pm,\uparrow}(\ddot{\fgl}_n) \subset \UUpm
\end{equation}
generated by $e_{[i;j)}^{(k)}$ (respectively $f_{[i;j)}^{(k)}$) for all $i < j$ and all $k \in \BN$. \\

\begin{proposition}
\label{prop:contains}

$U_{q,\oq}^{\pm,\uparrow}(\ddot{\fgl}_n)$ contains $e_{[i;j),k}$ (respectively $f_{[i;j),k}$) for all $k \in \BN$. \\

\end{proposition}

\begin{proof}

Under the isomorphisms $\Upsilon^\pm$, the Proposition reduces to a statement about the shuffle algebra. Specifically, we need to prove that:
\begin{equation}
\label{eqn:contains 1}
A_{[i;j)}^{z_i^k}
\end{equation}
lies in the algebra generated by:
\begin{equation}
\label{eqn:contains 2}
A_{[i;j)}^{\prod_{a=i}^{j-1} (z_a \oq^{\frac {2a}n})^{\left \lceil \frac {(a-i+1)k}{j-i} \right \rceil - \left \lceil \frac {(a-i)k}{j-i} \right \rceil}}
\end{equation}
for all $i<j$ and $k \in \BN$. We showed in \cite{Tor} (equations (1.1), (3.33), (3.77), although the elements denoted by $E$ therein are the antipodes of the elements denoted by $A$ herein) that, as $k$ varies over the entire $\BZ$, the elements \eqref{eqn:contains 2} generate a PBW basis of the shuffle algebra $\CS^+$: any element of the shuffle algebra can be written as a sum of products of the elements \eqref{eqn:contains 2}, taken in increasing order of the slope:
$$
\frac k{j-i} \in \BQ 
$$
If we simply take sums of products of the elements \eqref{eqn:contains 2} when the rational number above is positive (which corresponds to restricting attention to the case $k \in \BN$), then we may realize the subalgebra of $\CS^+$ of slope $>0$ elements, as defined in \cite{Tor}, Definition 3.12.\footnote{Note that \loccit considered subalgebras of elements of slope $< \mu$ for various rational numbers $\mu$, but the $>$ case is analogous, and can be obtained by applying $R(...,z_{ia},...) \mapsto R(...,z_{ia}^{-1},...)$} Therefore, all we need to do is to show that the elements \eqref{eqn:contains 1} have slope $>0$, which by the very definition of this notion in \loccitt, means that the rational function \eqref{eqn:contains 1} goes to 0 when we send any subset of its variable set $\{z_i,...,z_{j-1}\}$ to 0. This is an obvious consequence of \eqref{eqn:shuf} for $M = z_i^k$ with $k>0$. 
	
\end{proof}

\subsection{} 
\label{sub:asi}

Let us define the vector space: 
\begin{equation}
\label{eqn:upper}
\UUup^\br = U_{q,\oq}^{+,\uparrow}(\ddot{\fgl}_n) \otimes \fff [ \psi_{i,k}^+ ]^{1\leq i \leq n}_{k \in \BN} \otimes U_{q,\oq}^{-,\uparrow}(\ddot{\fgl}_n)
\end{equation}
and think of it as a subspace of $\UU^\br$. \\

\begin{proposition}
\label{prop:upper}

$\UUup^\br$ is a subalgebra of $\UU^\br$. \\

\end{proposition}

\begin{proof} The Proposition is equivalent to showing that:
\begin{equation}
\label{eqn:implies}
ba = \sum^{\text{various}}_{e \in U_{q,\oq}^{+,\uparrow}(\ddot{\fgl}_n), f \in U_{q,\oq}^{-,\uparrow}(\ddot{\fgl}_n)} e \cdot \prod_{\tilde{i}, \tilde{k}} \psi_{\tilde{i},\tilde{k}} \cdot f
\end{equation}
for any:
$$
a = e_{[i;j)}^{(k)} \qquad \text{and} \qquad b = f_{[i';j')}^{(k')}
$$
as well as the analogous statements where one of $a$ and $b$ is replaced by $\psi_{i,k}$ (which are easier, hence we leave them to the interested reader). We will prove this statement by induction on the total $|\text{horizontal degree}| = j-i+j'-i' \in \BN$. Recall the topological coproduct on the algebras $\CS^\pm$ defined in \cite{Tor}:
\begin{align*}
&\Delta (a) = a \otimes 1 + 1 \otimes a + \sum^{\text{various}}_{a_1 \in U_{q,\oq}^{+,\uparrow}(\ddot{\fgl}_n)} a_1 \otimes a_2 \\
&\Delta (b) = b \otimes 1 + 1 \otimes b + \sum^{\text{various}}_{b_1 \in U_{q,\oq}^{-,\uparrow}(\ddot{\fgl}_n)} b_1 \otimes b_2 
\end{align*} 
(we ignore certain Cartan elements from the coproduct, for brevity) where each $a_1$ and $b_1$ in the RHS has $|\text{horizontal degree}|$ strictly smaller than that of $a$ and $b$, respectively, in absolute value. Therefore, \eqref{eqn:drinfeld 2} implies:
\begin{equation}
\label{eqn:drinfeld 3}
ab = ba + \sum^{\text{various}}_{a_1, b_1 \in U_{q,\oq}^{\pm,\uparrow}(\ddot{\fgl}_n)} b_1a_1 \cdot \text{constants}
\end{equation}
The reason why $ab$ is the only term which features in the LHS is that the pairing of any element in $U_{q,\oq}^{+,\uparrow}(\ddot{\fgl}_n)$ with any element in $U_{q,\oq}^{-,\uparrow}(\ddot{\fgl}_n)$ vanishes for degree reasons, except if these elements are both multiples of $1$. Apart from $ba$, all terms in the RHS of \eqref{eqn:drinfeld 3} have $|\text{horizontal degree}|$ smaller than $j-i+j'-i'$. This immediately leads to the induction step necessary to prove \eqref{eqn:implies}.

\end{proof}

\subsection{} 
\label{sub:hat}

Let us consider the completion $\wUUup^\br$ of the algebra \eqref{eqn:top}, defined as:
\begin{equation}
\label{eqn:completion}
\wUUup^\br = \left\{ \sum^{e \in U_{q,\oq}^{+,\uparrow}(\ddot{\fgl}_n) \text{ with } \deg e = (\bd, k)}_{f \in U_{q,\oq}^{-,\uparrow}(\ddot{\fgl}_n) \text{ with } \deg f = (\bd', k')} e \cdot \prod_{\tilde{i}, \tilde{k}} \psi_{\tilde{i},\tilde{k}} \cdot f \right\}
\end{equation}
where any sum in the right-hand side may be infinite but with bounded above $k,k'$, and such that for any $N > 0$, all but finitely many summands have $|\bd|, |\bd'| > N$. By analogy with of Proposition 3.7 of \cite{W surf}, the completion \eqref{eqn:completion} is an algebra. The following elements of \eqref{eqn:completion} will play a very important role in the present paper:
\begin{equation}
\label{eqn:w gen}
W_{ij}^{k} = \sum_{s \leq \min(i,j)} (-1)^{j-s} q^{2\sigma_j - 2 \sigma_s + (j-s)(2k - 2r_j + 1)} \sum_{a+b+c = k + r_s - r_j}^{a \geq 1, b \geq 0, c \geq 1} e_{[s;i),a} \psi^+_{s,b} f_{[s;j),c}
\end{equation} 
\footnote{In the second sum above, we allow $a=0$ if and only if $s=i$ and $c=0$ if and only if $s=j$, in accordance with rule \eqref{eqn:rule}} defined for all $(i,j) \in \zzz$, $k \in \BZ$, where:
\begin{equation}
\label{eqn:sigma} 
\sigma_i = \begin{cases} r_1 + ... + r_i &\text{if } i > 0 \\ - r_{i+1} - ... - r_0 &\text{if } i \leq 0 \end{cases}
\end{equation}
The element \eqref{eqn:w gen} is not quite periodic in $i$ and $j$ modulo $n$, since we have:
\begin{equation}
\label{eqn:period w}
W_{ij}^k = W_{i-n,j-n}^k \cdot q^{r_1+...+r_n} \oq^{2r_j- 2k + \sigma_{j-1} - \sigma_i}
\end{equation}
for all $i,j,k$, but we will tacitly accept this abuse of notation. Note that $\deg W_{ij}^k = (-[i;j),k)$. Our reason for defining the elements \eqref{eqn:w gen} is the hope that the quotient:
\begin{equation}
\label{eqn:w alg}
\wUUup^\br \Big / \left( W_{ij}^{k} \right)_{(i,j) \in \zzz}^{k > r_j}
\end{equation}
is related to the (yet undefined) $q$--deformed $W$--algebra of type $\fgl_{r_1+...+r_n}$, associated to the nilpotent of Jordan type $\br$. The two algebras are isomorphic in the case of principal nilpotent (i.e. $n=1$), when the quotient \eqref{eqn:w alg} was shown in \cite{W} to match the $q$--deformed $W$--algebra that was rigorously defined in \cite{AKOS}, \cite{FF}, \cite{FR1}. We plan to study the case of rectangular nilpotent (i.e. $\br = (r,...,r)$) in \cite{Upcoming}. \\

\noindent For general Jordan type $\br$, our main reason for expecting that the quotient \eqref{eqn:w alg} to deserve being called a $q$--deformed $W$--algebra comes from geometry and physics: we will show in the next Section that the quotient \eqref{eqn:w alg} acts on the $K$--theory groups of moduli spaces of sheaves on $\BA^2$ with parabolic structure of type $\br$, as predicted by the AGT correspondence for gauge theory with surface operators. \\

\section{Moduli spaces of parabolic sheaves}

\subsection{}

\noindent Consider the surface $\BP^1 \times \BP^1$ and the divisors:
\begin{align*}
&D = \BP^1 \times \{0\} \\
&\infty = \BP^1 \times \{\infty\} \cup \{\infty\} \times \BP^1
\end{align*}
for a henceforth fixed choice of points $0,\infty \in \BP^1$. \\

\begin{definition}
\label{def:laumon}

Consider any $\bd, \br \in \nn$, and let $r = r_1+...+r_n$. A \textbf{parabolic sheaf} $\CF$ of degree $\bd$ and framing $\br$ is a flag of torsion-free rank $r$ sheaves: 
\begin{equation}
\label{eqn:flag0}
\CF_\bullet = \Big\{ \CF_n(-D) \subset \CF_1 \subset ... \subset \CF_{n-1} \subset \CF_n \Big\}
\end{equation}
on $\BP^1 \times \BP^1$ such that $c_2(\CF_i) = d_i [\emph{pt}]$, together with a collection of isomorphisms:
$$
\xymatrix{\CF_n(-D)|_{\infty} \ar[r] \ar[d]^\cong & \CF_1|_{\infty} \ar[r] \ar[d]^\cong &  ... \ar[r] \ar[d]^\cong & \CF_{n-1}|_{\infty} \ar[r] \ar[d]^\cong& \CF_n|_{\infty} \ar[d]^\cong \\
\CO_{\infty}^{r}(-D) \ar[r] & \CO_{\infty}^{r_1} \oplus \CO_{\infty}^{r-r_1}(-D) \ar[r] & ... \ar[r] & \CO_{\infty}^{r-r_n} \oplus \CO^{r_n}_{\infty}(-D) \ar[r] & \CO_{\infty}^{r}} \qquad 
$$
Note that the vertical isomorphisms force $c_1(\CF_i) = -(r_{i+1}+...+r_n)D$. We make \eqref{eqn:flag0} into an infinite flag, by setting $\CF_{i+n} = \CF_i(D)$ for all $i \in \BZ$. \\
\end{definition}

\noindent There exists a moduli space $\CM^{\br}_{\bd}$ parametrizing parabolic sheaves of degree $\bd$ and framing $\br$ (see \cite{B, BFG, HL}), which is a smooth, quasi-projective variety. \\

\subsection{}\label{sub:quiver}

The following picture represents the so-called chainsaw quiver (\cite{FR2, Nak 3})): \\

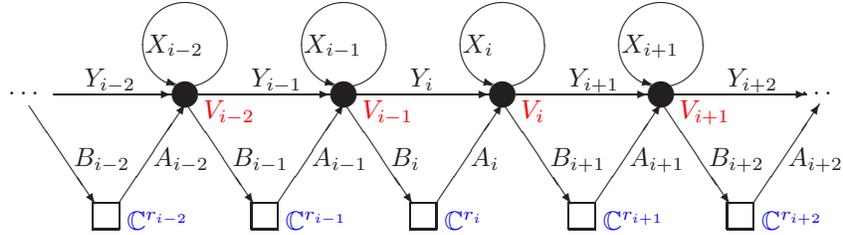
\begin{picture}(200,120)(35,-40)\label{pic:chainsaw}

\put(43,31){\dots}
\put(343,31){\dots}

\put(60,31){\vector(1,0){45}}
\put(72,34){$Y_{i-2}$}

\put(115,31){\vector(1,0){50}}
\put(135,34){$Y_{i-1}$}

\put(175,31){\vector(1,0){50}}
\put(195,34){$Y_{i}$}

\put(235,31){\vector(1,0){50}}
\put(255,34){$Y_{i+1}$}

\put(295,31){\vector(1,0){50}}
\put(315,34){$Y_{i+2}$}

\put(110,31){\circle*{10}}
\put(110,50){\circle{30}}
\put(95,47){$X_{i-2}$}
\put(117,22){$\color{red}{V_{i-2}}$}
\put(102,36){\vector(4,-1){5}}

\put(170,31){\circle*{10}}
\put(170,50){\circle{30}}
\put(155,47){$X_{i-1}$}
\put(177,22){$\color{red}{V_{i-1}}$}
\put(162,36){\vector(4,-1){5}}

\put(230,31){\circle*{10}}
\put(230,50){\circle{30}}
\put(215,47){$X_{i}$}
\put(237,22){$\color{red}{V_{i}}$}
\put(222,36){\vector(4,-1){5}}

\put(290,31){\circle*{10}}
\put(290,50){\circle{30}}
\put(275,47){$X_{i+1}$}
\put(297,22){$\color{red}{V_{i+1}}$}
\put(282,36){\vector(4,-1){5}}

\put(75,-20){\line(1,0){10}}
\put(75,-10){\line(1,0){10}}
\put(75,-20){\line(0,1){10}}
\put(85,-20){\line(0,1){10}}
\put(88,-20){$\color{blue}{\BC^{r_{i-2}}}$}

\put(51,27){\vector(2,-3){25}}
\put(85,-10){\vector(2,3){25}}
\put(68,4){$B_{i-2}$}
\put(98,4){$A_{i-2}$}

\put(135,-20){\line(1,0){10}}
\put(135,-10){\line(1,0){10}}
\put(135,-20){\line(0,1){10}}
\put(145,-20){\line(0,1){10}}
\put(148,-20){$\color{blue}{\BC^{r_{i-1}}}$}

\put(111,27){\vector(2,-3){25}}
\put(145,-10){\vector(2,3){25}}
\put(128,4){$B_{i-1}$}
\put(158,4){$A_{i-1}$}

\put(195,-20){\line(1,0){10}}
\put(195,-10){\line(1,0){10}}
\put(195,-20){\line(0,1){10}}
\put(205,-20){\line(0,1){10}}
\put(208,-20){$\color{blue}{\BC^{r_i}}$}

\put(171,27){\vector(2,-3){25}}
\put(205,-10){\vector(2,3){25}}
\put(188,4){$B_{i}$}
\put(218,4){$A_{i}$}

\put(255,-20){\line(1,0){10}}
\put(255,-10){\line(1,0){10}}
\put(255,-20){\line(0,1){10}}
\put(265,-20){\line(0,1){10}}
\put(268,-20){$\color{blue}{\BC^{r_{i+1}}}$}

\put(231,27){\vector(2,-3){25}}
\put(265,-10){\vector(2,3){25}}
\put(248,4){$B_{i+1}$}
\put(278,4){$A_{i+1}$}

\put(315,-20){\line(1,0){10}}
\put(315,-10){\line(1,0){10}}
\put(315,-20){\line(0,1){10}}
\put(325,-20){\line(0,1){10}}
\put(328,-20){$\color{blue}{\BC^{r_{i+2}}}$}

\put(291,27){\vector(2,-3){25}}
\put(325,-10){\vector(2,3){25}}
\put(308,4){$B_{i+2}$}
\put(338,4){$A_{i+2}$}

\end{picture} 

\noindent where it is understood that the black circles form a cycle of length $n$. More precisely, given $\bd = (d_1,...,d_n)$, let us fix vector spaces $V_1,...,V_n$ of dimensions $d_1,...,d_n$, respectively. Then consider the affine space of linear maps:
\begin{equation}
\label{eqn:affine}
M_\bd^{\br} = \bigoplus_{i=1}^{n} \Big[ \Hom(V_i,V_i) \oplus \Hom(V_{i-1},V_i) \oplus \Hom(\BC^{r_i},V_i) \oplus \Hom(V_{i-1}, \BC^{r_i}) \Big]
\end{equation}
where we identify $V_0$ with $V_{n}$. Elements of the vector space $M_\bd^\br$ will be quadruples $(X_i,Y_i,A_i,B_i)_{1 \leq i \leq n}$ of linear maps as in the Figure above. Consider the morphism:
\begin{equation}
\label{eqn:moment}
M_\bd^{\br} \stackrel{\mu}\longrightarrow \bigoplus_{i=1}^{n} \Hom(V_{i-1},V_i)
\end{equation}
$$
\mu(...,X_i,Y_i,A_i,B_i,...)_{1 \leq i \leq n} = \left(..., X_iY_i - Y_iX_{i-1}+A_iB_i,... \right)_{1 \leq i \leq n} 
$$
We let $GL_\bd := \prod_{i=1}^{n} GL(V_i)$ act on the vector space \eqref{eqn:affine} by conjugation, and let:
$$
M_{\bd}^{\br,\circ} \subset M_{\bd}^{\br}
$$
denote the open subset of quadruples $(X_i, Y_i, A_i, B_i)$ for which there is no collection of subspaces $\{\text{Im } A_i \subseteq W_i \subset V_i\}_{1 \leq i \leq n}$ preserved by the maps $\{X_i, Y_i\}_{1 \leq i \leq n}$. \\

\begin{proposition}

The moduli space of parabolic sheaves admits the presentation:
\begin{equation}
\label{eqn:adhm}
\CM_\bd^{\br} \cong (\mu^{-1}(0) \cap M_{\bd}^{\br,\circ}) / GL_\bd
\end{equation}

\end{proposition}

\noindent Therefore, we will freely go back and forth between describing points of $\CM_{\bd}^{\br}$ as either flags of sheaves $\CF_\bullet$ or quadruples $(X_i,Y_i,A_i,B_i)$ modulo conjugation by $GL_{\bd}$. The Proposition above is the reason why moduli spaces of framed parabolic sheaves are sometimes called ``chainsaw quiver varieties". \\

\subsection{}\label{sub:toract}

Let us fix a maximal torus $(\BC^*)^{r} \subset \prod_{i=1}^n GL(\BC^{r_i})$. The torus:
\begin{equation}
\label{eqn:torus}
T = \BC^* \times \BC^* \times (\BC^*)^{r}
\end{equation}
acts on $\CM_{\bd}^{\br}$ as follows: $(\BC^*)^{r}$ acts by multiplying the framing isomorphisms in Definition \ref{def:laumon}, while the other two $\BC^*$ factors act as multiplication on the base $\BP^1 \times \BP^1$ of parabolic sheaves. More precisely, we require the torus action to be the square the usual one, such that in terms of quadruples \eqref{eqn:moment}, the action is given by:
\begin{multline*}
(Q,\oQ,U_1,...,U_n) \cdot (..., X_i,Y_i,A_i,B_i,...) = \\ = \left(..., Q^2 X_i, \oQ^{2\delta_i^1} Y_i, A_i U^2_i, Q^2 \oQ^{2\delta_i^1} U_i^{-2} B_i,... \right)
\end{multline*}
for all $(Q,\oQ,U_1,...,U_n) \in (\BC^*)^{r} \times \BC^* \times \BC^*$. We will perceive each $U_i$ as a diagonal matrix acting on $\BC^{r_i}$ (corresponding to the standard basis), and denote its elementary characters by $u_{r_1+...+r_{i-1}+1},...,u_{r_1+...+r_i}$. Therefore, we have:
\begin{equation}
\label{eqn:coordinates}
K_{T}(\pt) = \text{Rep}(T) = \BZ \left[ q^{\pm 1}, \oq^{\pm 1}, u_1^{\pm 1},...,u_r^{\pm 1} \right] 
\end{equation}
where $q,\oq$ are the elementary characters of $\BC^* \times \BC^*$. We will study the localized $T-$equivariant algebraic $K-$theory groups corresponding to the action above:
$$
K^{\br}_\bd := K_{T} \left(\CM^{\br}_{\bd} \right) \bigotimes_{K_{T}(\pt)} \text{Frac}(K_{T}(\pt))
$$
Our main actor is the $\nn-$graded vector space:
\begin{equation}
\label{eqn:ktheory}
K^{\br} = \bigoplus_{\bd \in \nn} K^{\br}_\bd
\end{equation}
over the field $\text{Frac}(K_{T}(\pt)) = \BQ \left(q,\oq,u_1,...,u_r\right)$. \\

\subsection{} \label{sub:tangent}

As elements of $T$--equivariant $K$--theory, we have:
\begin{equation}
\label{eqn:frame}
[\CO^{r_i}] = \sum_{1\leq a \leq r}^{\wa = i} u_a^2 \ \in \ K^\br_\bd
\end{equation}
where we will write $\wa$ for that integer $i$ such that $r_1+...+r_{i-1} < a \leq r_1+...+r_i$. We may extend the notation $\wa$ to all integers $a$ via the formula below:
$$
\widehat{a+rk} = \widehat{a}+nk
$$
Moreover, we have the tautological vector bundles $\CV_i$ on $\CM^\br_\bd$, obtained by descending the vector spaces $V_i$ from Subsection \ref{sub:quiver} onto the quotient \eqref{eqn:adhm}. These vector bundles are equivariant, so they will give rise to elements of $K^\br_\bd$. The same thing is true for the universal sheaves $\CU_i$ on $\CM_\bd^\br$, defined by the property that their fibers over a point \eqref{eqn:flag0} are the restrictions to $(0,0) \in \BA^2$ of the quotients $\CF_i/\CF_{i-1}$, regarded as coherent sheaves on $D \cong \BP^1$. It is well-known that the universal sheaves have resolutions in terms of the tautological bundles, namely:
\begin{equation}
\label{eqn:universal}
[\CU_i] = [\CO^{r_i}] - (1-q^2) \left( [\CV_i] - [\CV_{i-1}] \right) \ \in \ K^\br_\bd
\end{equation}
for all $i \in \{1,...,n\}$, where we set:
\begin{align}
&\CV_{i+n} = \CV_{i} \cdot \oq^{-2} \label{eqn:change 1} \\
&u_{i+r} = u_{i} \cdot \oq^{-1} \label{eqn:change 2} 
\end{align}
for all $i\in \BZ$, which absorbs the dependence on the equivariant parameter $\oq$. \\

\begin{definition}
\label{def:tautological class}

To a Laurent polynomial $f(...,x_{ia},...)^{1 \leq i \leq n}_{a \in \BN}$ which is symmetric in the variables $x_{ia}$ for each $i$ separately, we will associate:
\begin{equation}
\label{eqn:tautological class}
\barf = f(...,\text{Chern roots of }[\CV_i],...) \in K^{\br}_{\bd}
\end{equation}
and call it a \textbf{tautological class} (although $\bd \in \nn$ is not part of the notation, it will either be implied from context or assumed to be arbitrary). \\

\end{definition}

\noindent For example, we have $[\CV_i] = \overline{X_i}$, where $X_i = \sum_{a \in \BN} x_{ia}$, and:
$$
[\CU_i] = \sum_{\wa = i} u_a^2 - (1-q^2) \left(\overline{X_i} - \overline{X_{i-1}} \right)
$$
where we make the convention that $X_{i+n} = X_i \cdot \oq^{-2}$ for all $i \in \BZ$, as in \eqref{eqn:change 1}. \\

\subsection{} \label{sub:fixed}

Because the moduli spaces $\CM^\br_\bd$ are smooth, the Thomason equivariant localization theorem states that restriction to the fixed locus induces an isomorphism:
\begin{equation}
\label{eqn:eq loc 1}
K^\br_\bd \quad \cong \bigoplus_{\bla \in (\CM_\bd^\br)^{T}} \BQ \left(q,\oq,u_1,...,u_r\right)|\bla \rangle 
\end{equation}
of $\BQ \left(q,\oq,u_1,...,u_r\right)$--vector spaces. The symbol $|\bla \rangle$ denotes the $K$--theory class of the skyscraper sheaf at the fixed point $\bla$, renormalized by the exterior algebra of the cotangent space at $\bla$, so that the isomorphism \eqref{eqn:eq loc 1} sends:
\begin{equation}
\label{eqn:eq loc 2}
c \in K^\br_\bd \quad \leadsto \quad \sum_{\bla \in (\CM_\bd^\br)^{T}} c|_\bla \cdot |\bla \rangle
\end{equation}
Therefore, to understand the isomorphism \eqref{eqn:eq loc 1}, we need to understand the restrictions of relevant $K$--theory classes to the torus fixed points. Combinatorially, the fixed points of $\CM_\bd^\br$ are indexed by $\br$--partitions, i.e. collections:
\begin{equation}
\label{eqn:fixed point}
\bla = (\lambda^1,...,\lambda^r) \qquad \text{where} \qquad \la^i = (\la^i_0 \geq \la^i_1 \geq ...)
\end{equation}
for all $i \in \{1,...,r\}$. We identify a partition $\lambda = (\lambda_0 \geq \lambda_1 \geq ... )$ with its Young diagram, i.e. the collection of $1 \times 1$ boxes whose bottom-left corner lies at all points $(x,y)$ with $0 \leq x < \lambda_y$. Similarly, the set of boxes of $\bla$ as in \eqref{eqn:fixed point} will refer to the disjoint union of the sets of boxes of the partitions $\lambda^1,...,\lambda^r$. We assign to the box:
$$
\sq = (x,y) \in \lambda^a  \qquad \begin{cases} \text{color } \wa + y \\ \text{weight } \chi_\sq = u_a q^{2x} \end{cases}
$$
The restriction of the tautological vector bundle $\CV_i$ to the fixed point $\bla$ is given by:
\begin{equation}
\label{eqn:restriction v}
\CV_i|_\bla = \sum_{a=1}^r \sum_{\sq = (x,y) \in \lambda^a}^{\wa+y \equiv i \text{ mod }n} \chi_\sq \cdot \oq^{2\frac {\wa+y-i}n}
\end{equation}
Therefore, we can explicitly describe the image of tautological classes under the identification \eqref{eqn:eq loc 1}. As a consequence of \eqref{eqn:eq loc 2}, this identification sends:
$$
\barf \quad \leadsto \quad \sum_{\br\text{--partitions } \bla} f(...,\chi_\sq \cdot \oq^{2\frac {\wa+y-i}n},...) |\bla \rangle
$$
where the expression $\chi_\sq \cdot \oq^{2\frac {\wa+y-i}n}$ is inserted into an argument of $f$ of the same color modulo $n$ as $\sq$. Note that there exist infinitely many $\br$--partitions, but the fixed points of $\CM^\br_\bd$ are only those $\br$--partitions of size $\bd$. These will be denoted by:
$$
\bla \vdash \bd
$$
which means that the number of boxes of $\bla$ of color $\equiv i$ mod $n$ is equal to $d_i$, $\forall i$. \\

\subsection{} 

Given two parabolic sheaves $\CF_\bullet$ and $\CF'_\bullet$ and $i \in \BZ$, we will write $\CF_\bullet \supset_\circ^i \CF'_\bullet$ if $\CF_j'$ is a colength $\delta_{\barj}^{\bari}$ subsheaf of $\CF_j$ for all $j\in \BZ$, with the quotient supported at the origin $\circ = (0,0) \in \BP^1 \times \BP^1$. Thus, we define the \textbf{simple correspondence} as:
\begin{equation}
\label{eqn:z}
\xymatrix{& \fZ_i \ar[ld]_{p_-} \ar[rd]^{p_+} & \\ \CM^{\br}_{\bd^-} & & \CM^{\br}_{\bd^+}} \qquad \fZ_i = \{\CF_\bullet \supset^i_\circ \CF_\bullet'\}
\end{equation}
for all $\bd_+ = \bd_- + \bs^i$. We also have the line bundle:
\begin{equation}
\label{eqn:l}
\xymatrix{ \CL_i \ar@{-->}[d] \\ \fZ_i}  \qquad \qquad \qquad \CL_i|_{(\CF_\bullet \supset_\circ^i \CF_\bullet')} = \Gamma(\BP^1 \times \BP^1, \CF_i/\CF_i')
\end{equation}
We claim that there exist isomorphisms:
\begin{equation}
\label{eqn:projectivizations}
\xymatrix{\BP(\CU_i) \ar[d] & \fZ_i \ar[ld]_{p_-} \ar[rd]^{p_+} \ar[r]^-{\cong} \ar[l]_-{\cong} & \BP \left(- q^2 \CU_{i+1}^\vee \right) \ar[d] \\ \CM_{\bd_-}^{\br} & & \CM_{\bd_+}^{\br}}
\end{equation}
such that $\CL_i$ is isomorphic to $\CO(1)$ and $\CO(-1)$ on the projectivizations $p_-$ and $p_+$, respectively. In \eqref{eqn:projectivizations}, if $\gamma$ denotes the $K$--theory class of a coherent sheaf $\Gamma$ of projective dimension 1, we abusively write $\BP(\gamma)$ for the dg scheme $\text{Proj}(\text{Sym}(\Gamma))$. We refer the reader to Definition 2.5, Proposition 2.8 and Proposition 2.10 of \cite{Surf} for a proof of \eqref{eqn:projectivizations} in the context of sheaves on surfaces (which corresponds to our $n=1$ case), and leave the general $n$ case as an analogous exercise to the interested reader. We note formulas (4.24) and (4.26) of \cite{Aff}, which show (in the general $n$, but $r_1=...=r_n = 1$ case) that the relative tangent space of the map $p_\pm$ has the same weights at the torus fixed points as prescribed by \eqref{eqn:projectivizations}. This equality of weights suffices for the $K$--theoretic computations performed in the present paper. \\

\subsection{} 

Consider the scheme $\fZ_{i,i+1} = \{\CF_\bullet \supset_{\circ}^i \CF'_\bullet \supset_\circ^{i+1} \CF''_\bullet\}$, which is endowed with line bundles $\CL_{i}$ and $\CL_{i+1}$ that parametrize the spaces of global sections of the length 1 quotients $\CF_i/\CF'_i$ and $\CF'_{i+1}/\CF''_{i+1}$, respectively. We have the projection maps:
$$
\xymatrix{& \fZ_{i,i+1} \ar[ld]_{\pi_-} \ar[rd]^{\pi_+} & \\ \fZ_i & & \fZ_{i+1}}
$$
that remember $\{\CF_\bullet \supset_\circ^{i} \CF'_\bullet\}$ and $\{\CF'_\bullet \supset_\circ^{i+1} \CF''_\bullet\}$, respectively. Let $\CU_{\bullet}'$ denote the universal sheaves from Subsection \ref{sub:tangent}, pulled back to either $\fZ_i$ or $\fZ_{i+1}$ via the projection map to the moduli space parametrizing the middle parabolic sheaf $\CF_\bullet'$. As in \eqref{eqn:projectivizations}, we claim that the maps $\pi_\pm$ also arise as projectivizations:
\begin{equation}
\label{eqn:projectivizations 2}
\xymatrix{\BP \left( \CU_{i+1}' + q^2\CL_i \right)  \ar[d] & \fZ_{i,i+1} \ar[ld]_-{\pi_-} \ar[rd]^-{\pi_+} \ar[r]^-{\cong} \ar[l]_-{\cong} & \BP \left( - q^2 (\CU_{i+1}')^\vee + q^2 \CL_{i+1}^\vee \right) \ar[d] \\ \fZ_i & & \fZ_{i+1}}
\end{equation}
such that the line bundle $\CL_{i+1}$ (respectively $\CL_i$) on $\fZ_{i,i+1}$ is isomorphic to $\CO(1)$ (respectively $\CO(-1)$) for the projectivizations above. We refer the reader to (4.25), (4.26) and Proposition 4.20 of \cite{W surf} for a proof of \eqref{eqn:projectivizations 2} in the context of sheaves on surfaces (which corresponds to our $n=1$ case), and leave the general $n$ case as an analogous exercise to the interested reader. As in the previous Subsection, we note that formulas (4.24) and (4.26) of \cite{Aff} show (in the general $n$, but $r_1=...=r_n = 1$ case) that the relative tangent space of the maps $\pi^\pm$ has the same weights at the torus fixed points as the projecitivization maps of \eqref{eqn:projectivizations 2}. \\

\noindent For general $i < j$, we define the so-called \textbf{fine correspondence}:
\begin{equation}
\label{eqn:dg}
\fZ_{i,...,j-1}
\end{equation} 
as  the derived fiber product of the diagram:
$$
\xymatrix{\fZ_{i,i+1} \ar[rd]^{\pi_+} & & \fZ_{i+1,i+2} \ar[ld]_{\pi_-} \ar[rd]^{\pi_+} & & ... \ar[ld]_{\pi_-} \ar[rd]^{\pi_+} & & \fZ_{j-2,j-1} \ar[ld]_{\pi_-} \\ & \fZ_{i+1} & & \fZ_{i+2} & & \fZ_{j-2} &} 
$$
Explicitly, the closed points of $\fZ_{i,...,j-1}$ are full flags of parabolic sheaves:
\begin{equation}
\label{eqn:full flag}
\{ \CF_\bullet \supset_\circ^i ... \supset_\circ^{j-1} \CF'_{\bullet} \}
\end{equation}
Let $\CL_i,...,\CL_{j-1}$ be the line bundles on $\fZ_{i,...,j-1}$ which parametrize the successive quotients in \eqref{eqn:full flag}. This derived fiber product is endowed with projection maps:
\begin{equation}
\label{eqn:zz}
\xymatrix{& \fZ_{i,...,j-1} \ar[ld]_{p_-} \ar[rd]^{p_+} & \\ \CM^{\br}_{\bd^-} & & \CM^{\br}_{\bd^+}} 
\end{equation}
which only remember $\CF_\bullet$ and $\CF'_\bullet$ of \eqref{eqn:full flag}, respectively. \\

\section{The algebra action}

\subsection{} We will now recall the action of the shifted quantum toroidal algebra on the $K$--theory of moduli spaces of parabolic sheaves (\cite{FT2}, although its precursors can be traced back to \cite{BFG}). Recall the equivariant parameters $u_1,...,u_r$ and the notation $\wa$ of Subsection \ref{sub:tangent}, and consider the following color-dependent Laurent polynomials:
\begin{align}
&\tau_+(z \text{ of color }i) = \prod_{\wa = i+1} \left( \frac {u_a}{q} - \frac {zq}{u_a}  \right) \label{eqn:tau+} \\ 
&\tau_-(z \text{ of color }i) = \prod_{\wa = i} \left( u_a - \frac z{u_a} \right) \label{eqn:tau-} 
\end{align}
Consider the operators:
$$
K^{\br}_{\bd} \xrightarrow{e_i^{(k)}} K^{\br}_{\bd + \vs^i}, \qquad K^{\br}_{\bd} \xrightarrow{f_i^{(k)}} K^{\br}_{\bd - \vs^i} 
$$
given by (the maps $p_\pm$ are as in \eqref{eqn:z}, or equivalently in \eqref{eqn:zz} for $j = i+1$):
\begin{align} 
&e_i^{(k)} (\alpha) = p_{+*} \Big(\CL_i^k \cdot p_-^{*}(\alpha) \Big) \cdot q^{d_{i+1} - d_i - 1 - r_{i+1}} \prod_{\wa = i+1} u_a\label{eqn:simple+} \\
&f_i^{(k)} (\alpha) = p_{-*} \Big(\CL_i^{k-r_i} \cdot p_+^{*}(\alpha) \Big) \cdot (-1)^{r_i + 1} q^{d_i - d_{i-1} - 2} \prod_{\wa = i} u_a \label{eqn:simple-} 
\end{align}
for any $\bd = (d_1,...,d_n)$. The maps $p_\pm$ are both proper and l.c.i., so $p_{\pm *}$ and $p_{\pm}^*$ are well-defined. In the formulas below, we will use the following notation:
$$
\zeta \left(\frac {x^{\pm 1}}{\chi^{\pm 1}_{\bla}} \right)^{\pm 1} = \prod_{\sq \in \bla} \zeta \left(\frac {x^{\pm 1}}{\chi^{\pm 1}_{\sq}} \right)^{\pm 1}
$$
for any $\br$--partition $\bla$, while for any (complexes of) vector bundles $\CE$ and $\CE'$, set:
$$
\frac {\CE'}{\CE} = [\CE'] \otimes [\CE^\vee]
$$
Then in the fixed point basis, the operators \eqref{eqn:simple+} and \eqref{eqn:simple-} take the form:
$$
\langle \bla | e_i^{(k)} | \bmu \rangle = (1-q^2) \chi_\bsq^k \wedge^\bullet \left(\frac {\chi_\bsq q^2}{\CU_{i+1}|_\bmu} \right) \cdot q^{d^\bmu_{i+1} - d^\bmu_i-1-r_{i+1}} \prod_{\wa = i+1} u_a = 
$$
\begin{equation}
\label{eqn:simple loc+}
= (q^{-1}-q) \chi_\bsq^k \tau_+(\chi_\bsq) \zeta \left(\frac {\chi_\bsq}{\chi_{\bmu}} \right)
\end{equation}
and:
$$
\langle \bmu | f_i^{(k)} | \bla \rangle = (1-q^2) \chi_\bsq^{k-r_i} \wedge^\bullet \left( - \frac {\CU_i|_\bla}{\chi_\bsq } \right) \cdot (-1)^{r_i + 1} q^{d^\bla_i - d^\bla_{i-1} - 2} \prod_{\wa = i} u_a =
$$
\begin{equation}
\label{eqn:simple loc-}
= (1-q^{-2}) \chi_\bsq^k \left[ \tau_-(\chi_\bsq) \zeta \left(\frac {\chi_\bla}{\chi_\bsq} \right) \right]^{-1}
\end{equation}
where $\bsq = \bla \backslash \bmu$, and $d^\bla, d^\bmu$ denote the degree vectors of the torus fixed points $\bla,\bmu$. \\

\subsection{} As a consequence of \eqref{eqn:universal} and \eqref{eqn:restriction v}, we have:
$$
\det \CU_i|_\bla = q^{2(d_i-d_{i-1})} \prod_{\wa = i} u_a^2
$$
With this in mind, let us also define the operators:
\begin{align}
&\psi^+_i(z) = \text{multiplication by } \wedge^\bullet \left( \frac {\CU_i}z \right) \cdot \frac {(-1)^{r_i}}{\sqrt {\det \CU_i}} \cdot q^{\sigma_i} \label{eqn:psi+} \\ 
&\psi^-_i(z) = \text{multiplication by } \wedge^\bullet \left( \frac z{\CU_i} \right) \cdot \sqrt {\det \CU_i} \cdot q^{-\sigma_i} \label{eqn:psi-}
\end{align}
(expanded in powers of $z^{\mp 1}$) which are diagonal in the basis $|\bla\rangle$:
\begin{align*} 
&\langle \bla | \psi^+_i(z) | \bmu \rangle = \delta_\bla^\bmu q^{\sigma_i} (-1)^{r_i} \prod^{\wa \equiv i}_{1\leq a \leq r} \left( \frac 1{u_a} - \frac {u_a}z \right) \zeta \left( \frac {\chi_\bla}z \right) \\
&\langle \bla | \psi^-_i(z) | \bmu \rangle = \delta_\bla^\bmu q^{-\sigma_i} \prod^{\wa \equiv i}_{1\leq a \leq r} \left(u_a - \frac z{u_a} \right) \zeta \left( \frac {\chi_\bla}z \right)
\end{align*}
where $z$ is thought to have color $i$. \\

\begin{theorem}
\label{thm:act}

Formulas \eqref{eqn:simple+}, \eqref{eqn:simple-}, \eqref{eqn:psi+}, \eqref{eqn:psi-} give an action $\UU^{\br} \curvearrowright K^{\br}$. \\

\end{theorem}

\noindent In equivariant cohomology, Theorem \ref{thm:act} was proved in Theorem 12.17 of \cite{FT2}, and the $K$--theoretic case that we are concerned with is proved similarly. Alternatively, the reader may find the $r_1=...=r_n=1$ case of Theorem \ref{thm:act} proved in Theorems 1.2 and 3.17 of \cite{Aff} (see also Proposition 3.27 of \loccit for formulas analogous to \eqref{eqn:simple loc+} and \eqref{eqn:simple loc-}). The case of general $r_1,...,r_n$ is a straightforward generalization. \\

\subsection{} 

We will now consider the more complicated cousins of the operators studied in the previous subsection, which were introduced in \cite{Aff} in the case $\br = (1,...,1)$. For any pair of integers $i < j$, we consider the maps as in diagram \eqref{eqn:zz} and define the following operators for any Laurent polynomial $M(z_i,...,z_{j-1})$:
\begin{equation}
\label{eqn:operators}
K^{\br}_{\bd^-} \xrightarrow{e_{[i;j)}^M} K^{\br}_{\bd^+}, \qquad \qquad K^{\br}_{\bd^+} \xrightarrow{f_{[i;j)}^M} K^{\br}_{\bd^-},
\end{equation}
\begin{align}
&e_{[i;j)}^{M}(\alpha) = p_{+*} \Big( M(\CL_i,...,\CL_{j-1}) \cdot p_-^{*}(\alpha) \Big) \cdot q^{d^+_j - d^-_i} \prod_{a=i}^{j-1} \frac {\prod_{\wb = a+1} u_b}{q^{r_{a+1}+1}} \label{eqn:fine+} \\
&f_{[i;j)}^{M}(\alpha) = p_{-*} \Big(\frac {M(\CL_i,...,\CL_{j-1})}{\CL_i^{r_i} ... \CL_{j-1}^{r_{j-1}}} \cdot p_+^{*} (\alpha) \Big)  \cdot q^{d^+_{j-1} - d_{i-1}^-} \prod_{a=i}^{j-1} \frac {\prod_{\wb = a} u_b}{(-1)^{r_a+1} q^2} \label{eqn:fine-}
\end{align}
Above, the morphisms $p_{\pm *}$ and $p_\pm^*$ are well-defined because all the maps $\pi_\pm$ in the zig-zag diagram underneath \eqref{eqn:dg} are proper and l.c.i. (and in turn, this is because of the isomorphism \eqref{eqn:projectivizations 2}, and the fact that the $\BP$'s that appear in the top row of this formula are regular subschemes of projective bundles over the varieties in the bottom row of this formula; see Section 2.3 of \cite{Surf} for context). \\

\noindent Almost word for word as in the proof of Proposition 4.23 of \cite{Aff} (see equation (5.13) of \loccitt), one can show that the operators \eqref{eqn:fine+} and \eqref{eqn:fine-} have the following matrix coefficients in the fixed point basis:
$$
\langle \bla | e_{[i;j)}^{M} | \bmu \rangle = \prod_{\bsq \in \blamu} (q^{-1} - q) \tau_+(\chi_\bsq) \zeta \left(\frac {\chi_\bsq}{\chi_{\bmu}} \right) \sum_{\syt} \frac {M(\chi_i,...,\chi_{j-1}) \prod_{i \leq a < b < j} \zeta \left( \frac {\chi_b}{\chi_a} \right)}{\left(1 - \frac {q^2\chi_i}{\chi_{i+1}} \right) ... \left(1 - \frac {q^2\chi_{j-2}}{\chi_{j-1}} \right)}
$$
$$
\langle \bmu | f_{[i;j)}^{M} | \bla \rangle = \prod_{\bsq \in \blamu} (1-q^{-2}) \left[ \tau_-(\chi_\bsq) \zeta \left(\frac {\chi_\bla}{\chi_\bsq} \right) \right]^{-1} \sum_{\syt} \frac {M(\chi_i,...,\chi_{j-1}) \prod_{i \leq a < b < j} \zeta \left( \frac {\chi_b}{\chi_a} \right)}{\left(1 - \frac {q^2\chi_i}{\chi_{i+1}} \right) ... \left(1 - \frac {q^2\chi_{j-2}}{\chi_{j-1}} \right)}
$$
where for any $\bmu \subset \bla$, a standard Young tableau (abbreviated SYT) is any labeling of the boxes of $\blamu$ with the integers $i,...,j-1$ such that the box labeled with $a$ has color $\equiv a$ modulo $n$ and weight $\chi_a$, and no box is directly above or to the right of any other box with greater label. If there are no such standard Young tableaux for given $\bla, \bmu$, or if $\bmu \not \subset \bla$, the right-hand sides of the expressions above are set to 0. \\

\begin{proposition}
\label{prop:act}
	
Under the action of Theorem \ref{thm:act}, the elements \eqref{eqn:elements} act on $K^\br$ as the operators \eqref{eqn:operators} with the same name, hence the abuse of notation. \\
	
\end{proposition}

\noindent The Proposition above was proved in Theorem 1.2 of \cite{Aff} in the case $\br = (1,...,1)$, and the proof in the general case is almost identical, hence we leave it to the interested reader. With Proposition \ref{prop:act} in mind, it makes sense to consider the following particular instances of the operators \eqref{eqn:operators}, which are induced by the action of the elements \eqref{eqn:el k} and \eqref{eqn:el mu} under Theorem \ref{thm:act}:
\begin{equation}
\label{eqn:op k}
e_{[i;j),k} = e_{[i;j)}^{M(z_i,...,z_{j-1}) \mapsto z_i^k} 
\end{equation} 
\begin{equation}
\label{eqn:op mu}
e^{(k)}_{[i;j)} = e_{[i;j)}^{M(z_i,...,z_{j-1}) \mapsto \prod_{a=i}^{j-1} \left( z_a \oq^{\frac {2a}n} \right)^{\left\lceil \frac {(a-i+1)k}{j-i} \right\rceil - \left\lceil \frac {(a-i)k}{j-i} \right\rceil}}
\end{equation}
for all $i < j$ and $k \in \BZ$, as well as the analogous notations for the $f$'s. We will also write $e_{[i;i),k} = f_{[i;i),k} = \delta_k^0$, to match with the convention \eqref{eqn:rule}. \\

\subsection{} Given integers $i,j$, let us consider the following dg scheme for all $s \leq \min(i,j)$:
\begin{equation}
\label{eqn:fv}
\fV_{i;s;j} = \fZ_{s,...,i-1} \underset{\CM^{\br}_{\bd}}{\times} \fZ_{s,...,j-1}
\end{equation}
Its closed points are given by:
\begin{equation}
\label{eqn:fv closed}
\fV_{i;s;j} = \Big\{
\tCF_\bullet' \subset_\circ^{i-1} ... \subset_\circ^{s+1} \CF_\bullet' \subset_\circ^s \CF_\bullet \supset_\circ^s \CF_\bullet'' \supset_\circ^{s+1} ... \supset_\circ^{j-1} \tCF_\bullet'' \Big\}
\end{equation}
where the moduli space $\CM^\br_{\bd}$ that features in \eqref{eqn:fv} is the one which parametrizes the parabolic sheaf $\CF_\bullet$. Let us consider the following line bundles on $\fV_{i;s;j}$:
\begin{equation}
\label{eqn:notation 1}
\CL'_{i-1},...,\CL'_s, \CL''_s,..., \CL''_{j-1}
\end{equation}
which parametrize the successive one-dimensional quotients of the inclusions in \eqref{eqn:fv closed} (read left-to-right). Also consider the following universal sheaves on $\fV_{i;s;j}$:
\begin{equation}
\label{eqn:notation 2}
\tCU_s',..., \CU_s', \CU_s, \CU_s'',..., \tCU_s''
\end{equation}
which keep track of the parabolic sheaves in \eqref{eqn:fv closed}, respectively. We have maps:
\begin{equation}
\label{eqn:correspondence final}
\xymatrix{& \fV_{i;s;j} \ar[ld]_{\pi'_s} \ar[rd]^{\pi''_s} \\ 
\CM^\br_{\tbd'} & & \CM^\br_{\tbd''}} 
\end{equation}
In terms of closed points \eqref{eqn:fv closed}, $\pi'_s$, $\pi''_s$ remember the parabolic sheaves $\tCF'_\bullet$, $\tCF''_\bullet$, respectively. As a consequence of Proposition \ref{prop:act}, we conclude that the element:
$$
W_{ij}^{k} \in \UU^\br 
$$ 
of \eqref{eqn:w gen} acts on $K^\br$ by the correspondence:
\begin{equation}
\label{eqn:correspondence}
W_{ij}^k = \sum_{s \leq \min(i,j)} \pi'_{s*} \left( \sum^{a \geq 1, b \geq 0, c \geq 1}_{a+b+c = k + r_s - r_j} \frac {{\CL_s'}^{a} \cdot (-1)^b\wedge^b(\CU_s) \cdot  {\CL_s''}^{c}}{(\CL_s'')^{r_s}...(\CL_{j-1}'')^{r_{j-1}}} \cdot {\pi_s''}^* \right) \cdot 
\end{equation}
$$
(-1)^{\sum_{s < t < j} r_s} q^{2\sigma_j  - \sigma_i - i - j + \tilde{d}'_i + \tilde{d}''_{j-1} + (j-s)(2k - 2r_j) + 2s - 2d_s} \prod_{s < t \leq i}^{\wb = t} u_b  \prod_{s < t \leq j-1}^{\wb = t} u_b 
$$
\footnote{We make the same convention as in the footnote directly after \eqref{eqn:w gen}: if $s = i$ (respectively $s = j$), the corresponding summand in the formula above must have $a=0$ (respectively $c = 0$)} where we write $\tbd',...,\bd',\bd, \bd'',...,\tbd''$ for the degree vectors of the parabolic sheaves in \eqref{eqn:fv closed}, read left-to-right. Theorem \ref{thm:main} is equivalent to the fact that the right-hand side of \eqref{eqn:correspondence} is 0 as an endomorphism of $K^\br$, for all natural numbers $k>r_j$, which we will now prove. \\

\begin{proof} \emph{of Theorem \ref{thm:main}:} Let us consider the dg scheme:
\begin{equation}
\label{eqn:fw} 
\fW_{i;s;j} = \fZ_{s-1,...,i-1} \underset{\fZ_{s-1}}{\times} \fZ_{s-1,...,j-1}
\end{equation} 
whose closed points take the form:
\begin{equation}
\label{eqn:extended flag}
\xymatrix{& {\widetilde{\CF}}_\bullet & \\
& {\widetilde{\CF}}_\bullet' \subset_\circ^{i-1} ... \subset_\circ^{s+1} \CF_\bullet' \subset_\circ^s \CF_\bullet \supset_\circ^s \CF_\bullet'' \supset_\circ^{s+1} ... \supset_\circ^{j-1} {\widetilde{\CF}}_\bullet'' \ar@{^{(}->}[u]_\circ^{s-1} &} 
\end{equation}
We will write $\CL_{s-1}$ for the line bundle on $\fW_{i;s;j}$ which parametrizes the one-dimensional quotient corresponding to the vertical inclusion, and apply the notations \eqref{eqn:notation 1} and \eqref{eqn:notation 2} to $\fW_{i;s;j}$ as well. Consider the maps:
$$
\xymatrix{\fW_{i;s+1;j} \ar[rd]_{\delta_s} & & \fW_{i;s;j} \ar[ld]^{\e_s} \\ & \fV_{i;s;j} &}
$$
where $\e_s$ forgets $\tCF_\bullet$, while $\delta_s$ is the embedding of the locus $\CF_\bullet' = \CF_\bullet''$. Lemma 6.2 of \cite{W surf} states that we have the following equality of $K$--theory classes on $\fV_{i;s;j}$: \footnote{Note that \loccit actually proves \eqref{eqn:above} in the case when $i = j =s+1$, but the general case is obtained by pull-back under the natural forgetful map $\fV_{i;s;j} \rightarrow \fV_{s+1;s;s+1}$}
\begin{equation}
\label{eqn:above}
\delta_{s*}(\CL_s^k) - \e_{s*}(\CL_{s-1}^{k-r_s}) (-1)^{r_s} q^{2(k-r_s)} \det \CU_s = \int_{\infty - 0} \frac {y^k \wedge^\bullet \left( \frac {\CU_s}y \right)}{\left(1-\frac y{\CL_s'}\right)\left(1-\frac y{\CL_s''}\right)}
\end{equation}
for all $k \in \BZ$, where we write:
$$
\int_{\infty - 0} f(y) = \underset{y = \infty}{\text{Res}} \frac {f(y)}y - \underset{y = 0}{\text{Res}} \frac {f(y)}y
$$
Let us plug $k = k+r_s-r_j$ and $\det \CU_s = q^{2d_s - 2 d_{s-1}} \prod_{\wb = s} u_b^2$ into \eqref{eqn:above}:
$$
\delta_{s*}(\CL_s^{k+r_s-r_j}) - \e_{s*}(\CL_{s-1}^{k-r_j}) (-1)^{r_s} q^{2(k-r_j+d_s - d_{s-1})} \prod_{\wb = s} u_b^2 = \int_{\infty - 0} \frac {y^{k+r_s-r_j} \wedge^\bullet \left( \frac {\CU_s}y \right)}{\left(1-\frac y{\CL_s'}\right)\left(1-\frac y{\CL_s''}\right)}
$$
Finally, let us multiply both sides of the equation above by:
\begin{equation}
\label{eqn:as} 
A_s = \frac {\text{second row of \eqref{eqn:correspondence}}}{(\CL_s'')^{r_s}...(\CL_{j-1}'')^{r_{j-1}}}
\end{equation}  
As a consequence of the elementary identity:
\begin{equation}
\label{eqn:as as} 
A_s = A_{s-1} \cdot \frac {(-1)^{r_s} \CL_{s-1}^{r_{s-1}}}{q^{2(k-r_j+d_s-d_{s-1})} \prod_{\wb = s} u_b^2} 
\end{equation}
\footnote{The apparent discrepancy between the powers of $q$ in the left and right-hand sides of \eqref{eqn:as as} is due to the fact that on the second row of \eqref{eqn:correspondence}, the number $d_s$ refers to the degree of the middle flag of sheaves in \eqref{eqn:fv closed}; this number for $A_{s-1}$ is 1 less than this number for $A_s$} we conclude that:
\begin{equation}
\label{eqn:stevie}
\delta_{s*}(\CL_s^{k+r_s-r_j} A_s) - \e_{s*}(\CL_{s-1}^{k+r_{s-1}-r_j}A_{s-1}) =
\end{equation}
$$
= \int_{\infty - 0} \frac {y^{k+r_s-r_j} \wedge^\bullet \left( \frac {\CU_s}y \right)}{\left(1-\frac y{\CL_s'}\right)\left(1-\frac y{\CL_s''}\right)} \cdot \frac {\text{second row of \eqref{eqn:correspondence}}}{(\CL_s'')^{r_s}...(\CL_{j-1}'')^{r_{j-1}}}
$$
The equality above also holds for $s = \min(i,j)$, with the distinction that in this case the term $\delta_{s*}$ in the LHS vanishes, and the term $1 - \frac y{\CL_s'}$ (respectively $1 - \frac y{\CL_s''}$ ) in the RHS is replaced by 1 if $i \leq j$ (respectively if $i \geq j$). For fixed $\tbd'$ and $\tbd''$, we may interpret \eqref{eqn:stevie} as an equality of correspondences via the maps in \eqref{eqn:correspondence final}, i.e.:
\begin{multline}
\pi'_{s*} \Big( \delta_{s*}(\CL_s^{k+r_s-r_j} A_s) \cdot {\pi''_s}^* \Big) - \pi'_{s*} \Big( \e_{s*}(\CL_{s-1}^{k+r_{s-1}-r_j}A_{s-1}) \cdot {\pi''_s}^* \Big) =
 \\ = \pi'_{s*} \left( \int_{\infty - 0} \frac {y^{k+r_s-r_j} \wedge^\bullet \left( \frac {\CU_s}y \right)}{\left(1-\frac y{\CL_s'}\right)\left(1-\frac y{\CL_s''}\right)} \cdot \frac {\text{second row of \eqref{eqn:correspondence}}}{(\CL_s'')^{r_s}...(\CL_{j-1}'')^{r_{j-1}}} \cdot {\pi''_s}^* \right) 
\label{eqn:fire}
\end{multline}
as an equality of operators $K_T(\CM^\br_{\tbd''}) \rightarrow K_T(\CM^\br_{\tbd'})$. Consider the generating series:
\begin{equation}
\label{eqn:series e final}
e_{[s;i)}(y) = - \sum_{k=1}^\infty \frac {e_{[s;i),k}}{y^k} = p_{+*} \left[ \frac 1{1- \frac y{\CL_s'}} \cdot p_-^{*} \right] q^{\td'_i - d_s} \prod_{a=s}^{i-1} \frac {\prod_{\wb = a+1} u_b}{q^{r_{a+1}+1}} 
\end{equation}
\begin{equation}
\label{eqn:series f final}
f_{[s;j)}(y) = - \sum_{k=1}^\infty \frac {f_{[s;j),k}}{y^k} = 
\end{equation}
$$
= p_{-*} \left[\frac {1}{\left(1- \frac y{\CL_s''}\right)(\CL''_s)^{r_s} ... (\CL''_{j-1})^{r_{j-1}}} \cdot p_+^{*} \right]  q^{\td''_{j-1} - d_{s-1}} \prod_{a=s}^{j-1} \frac {\prod_{\wb = a} u_b}{(-1)^{r_a+1} q^2} 
$$
of the operators \eqref{eqn:op k}, with the notations as in \eqref{eqn:fine+} and \eqref{eqn:fine-}. Thus:
$$
\text{RHS of \eqref{eqn:fire}} = (-1)^{j-s} q^{2\sigma_j - 2 \sigma_s + (j-s)(2k - 2r_j + 1)} \int_{\infty - 0} y^{k+r_s-r_j}  e_{[s;i)}(y) \psi_s^+(y) f_{[s;j)}(y) 
$$
Let us sum the left-hand side of \eqref{eqn:fire} over all integers $s \leq \min(i,j)$, which is a well-defined operation because only finitely many terms in the sum over $s$ are non-zero for any fixed $\tbd'$ and $\tbd''$. \\

\begin{claim}
\label{claim:delta epsilon}

In the left-hand side of \eqref{eqn:fire}, the summand involving $\delta_{s*}$ precisely cancels out the summand involving $\e_{s+1*}$, for any $s$. \\

\end{claim}

\noindent We will prove Claim \ref{claim:delta epsilon} at the end of the paper, but let us first continue the proof of Theorem \ref{thm:main}. Using the Claim, we conclude that:
\begin{multline}
0 = \sum_{s \leq \min(i,j)} (-1)^{j-s} q^{2\sigma_j - 2 \sigma_s + (j-s)(2k - 2r_j + 1)} \\ \cdot \int_{\infty - 0} y^{k+r_s-r_j}  e_{[s;i)}(y) \psi_s^+(y) f_{[s;j)}(y) \label{eqn:sum}
\end{multline}
The integral in \eqref{eqn:sum} consists of two terms: the residue at $\infty$ and the residue at $0$, and formula \eqref{eqn:sum} states that they are equal. The residue at $y = \infty$ is equal to:
$$
W_{ij}^k = \sum_{s \leq \min(i,j)} (-1)^{j-s} q^{2\sigma_j - 2 \sigma_s + (j-s)(2k - 2r_j + 1)}  \sum_{a+b+c = k + r_s - r_j}^{a \geq 1, b \geq 0, c \geq 1} e_{[s;i),a} \psi^+_{s,b} f_{[s;j),c} 
$$
Meanwhile, the residue at $0$ vanishes if $k >  r_j$. This is because $e_{[s;i)}(y)$ and $f_{[s;j)}(y)$ are regular at $y = 0$ (as is apparent from formulas \eqref{eqn:series e final} and \eqref{eqn:series f final}), while $\psi_s^+(y)$ has a pole of order exactly $r_s$ (as is apparent from formula \eqref{eqn:psi+}). This implies that $W_{ij}^k = 0$ if $k > r_j$, thus completing the proof of the Theorem. \\

\begin{proof} \emph{of Claim \ref{claim:delta epsilon}:} We have the equality:
$$
\pi'_{s*} \Big( \delta_{s*}(\CL_{s}^{k+r_{s}-r_j} A_s)  {\pi_s''}^* \Big) = \rho'_* \Big(\CL_{s}^{k+r_{s}-r_j}A_{s}   {\rho''}^* \Big) = \pi'_{s+1*} \Big( \e_{s+1*}(\CL_{s}^{k+r_{s}-r_j}A_{s})  {\pi''_{s+1}}^* \Big)
$$
where we consider the maps:
$$
\xymatrix{& \fW_{i;s+1;j} \ar[ld]_{\rho'} \ar[rd]^{\rho''} \\ 
\CM^\br_{\tbd'} & & \CM^\br_{\tbd''}} 
$$
In terms of \eqref{eqn:extended flag}, $\rho'$, $\rho''$ remember the parabolic sheaves $\tCF'_\bullet$, $\tCF''_\bullet$, respectively.

\end{proof} 

\end{proof}

\end{document}